\documentclass{amsart}

\usepackage[latin1]{inputenc}
\usepackage[english]{babel}
\usepackage{indentfirst}
\usepackage{hyperref}
\usepackage{amssymb}
\usepackage{amsthm}
\usepackage{xcolor}
\usepackage[all]{xy}
\usepackage[mathscr]{eucal}

\def\epsilon{\varepsilon}
\newtheorem{theo}{Theorem}[section]
\newtheorem{theorem}[theo]{Theorem}
\newtheorem{proposition}[theo]{Proposition}
\newtheorem{corollary}[theo]{Corollary}
\newtheorem{lemma}[theo]{Lemma}

\theoremstyle{definition}
\newtheorem{definition}[theo]{Definition}
\newtheorem{example}[theo]{Example}
\newtheorem{remark}[theo]{Remark}

\newcommand{\pten}{\ensuremath{\widehat{\otimes}_\pi}}
\newcommand{\Lip}{{\mathrm{Lip}}_0}

\numberwithin{equation}{section}

\date{\today}

\title{The Daugavet property in spaces of vector-valued Lipschitz functions}

\author{Abraham Rueda Zoca}
\address{Universidad de Murcia, Departamento de Matem\'{a}ticas, Campus de Espinardo 30100 Murcia, Spain
	\newline
	\href{https://orcid.org/0000-0003-0718-1353}{ORCID: \texttt{0000-0003-0718-1353} }}
\email{\texttt{abraham.rueda@um.es}}
\urladdr{\url{https://arzenglish.wordpress.com}}

\thanks{The research of Abraham Rueda Zoca was supported by MCIN/AEI/10.13039/501100011033: Fellowship Juan de la Cierva-Formaci\'on FJC2019-039973 and grants MTM2017-86182-P and PGC2018-093794-B-I00, by Fundaci\'on S\'eneca: ACyT Regi\'on de Murcia grant 20797/PI/18, and by Junta de Andaluc\'ia: Grants A-FQM-484-UGR18 and by FQM-0185.}



\begin{document}

\maketitle
\markboth{ABRAHAM RUEDA ZOCA}{DAUGAVET PROPERTY IN SPACES OF VECTOR-VALUED LIPSCHITZ FUNCTIONS}

\begin{abstract}
We prove that if a metric space $M$ has the finite CEP then $\mathcal F(M)\pten X$ has the Daugavet property for every non-zero Banach space $X$. This applies, for instance, if $M$ is a Banach space whose dual is isometrically an $L_1(\mu)$ space. If $M$ has the CEP then $L(\mathcal F(M),X)=\Lip(M,X)$ has the Daugavet property for every non-zero Banach space $X$, showing that this is the case when $M$ is an injective Banach space or a convex subset of a Hilbert space. 
\end{abstract}

\section{Introduction}

A Banach space $X$ is said to have the Daugavet property if every rank-one operator $T:X\longrightarrow X$ satisfies the equality
\begin{equation}\label{ecuadauga}
\Vert T+I\Vert=1+\Vert T\Vert,
\end{equation}
where $I$ denotes the identity operator. The previous equality is known as \emph{Daugavet equation} because I.~Daugavet proved in \cite{dau} that every compact operator on $\mathcal C([0,1])$ satisfies \eqref{ecuadauga}. Since then, many examples of Banach spaces enjoying the Daugavet property have appeared such as $\mathcal C(K)$ for a perfect compact Hausdorff space $K$; $L_1(\mu)$ and $L_\infty(\mu)$ for a non-atomic measure $\mu$; or preduals of Banach spaces with the Daugavet property (see \cite{kkw,kssw,werner} and references therein for a detailed treatment of the Daugavet property).

One of the problem that has focused the attention of many experts on the Daugavet property and on Lipschitz free spaces is to determine when a space of Lipschitz functions $\Lip(M)$ enjoys the Daugavet property. This problem was motivated by the question whether the space $\Lip([0,1]^2)$ enjoys the Daugavet property, posed in \cite[Section 6, Question (1)]{werner}. As consequence of a colective effort during one decade \cite{am,gpr18,ikw}, a characterisation of those metric spaces $M$ with the Daugavet property is given: $\Lip(M)$ enjoys the Daugavet property if, and only if, $M$ is length and if, and only if, the Lipschitz free space $\mathcal F(M)$ has the Daugavet property, which is in turn equivalent to the fact that the unit ball of $\mathcal F(M)$ does not have any strongly exposed point.

A natural question at this point, asked in \cite{gpr18,lr2020} is to determine when a vector-valued space of Lipschitz functions $\Lip(M,X)$, for a Banach space $X$, has the Daugavet property. This problem has shown to be surprisingly hard, and no result in this line has appeared apart from some particular examples like the space $\Lip(H,H)$ for a Hilbert space $H$ \cite[P. 482]{gpr18}.

The reason why this problem is difficult is not so surprising, however, when we have a closer look to the Banach space structure of the space $\Lip(M,X)$. In general, given a metric space $M$ and a Banach space $X$, it is know (see details below) that there exists an isometric identification between the spaces $\Lip(M,X)$ and the space of bounded linear operators $L(\mathcal F(M),X)$. It the case of $X$ begin the dual of a Banach space $Y$, we know that $L(\mathcal F(M),Y^*)=(\mathcal F(M)\pten Y)^*$, and let us mention that it is an open question from \cite[Section 6, Question (3)]{werner} whether the Daugavet property is stable by taking projective tensor products. Observe that, in general, given two Banach spaces $X$ and $Y$, it is known that $X\pten Y$ may fail the Daugavet property if we only require the Daugavet property in either $X$ or $Y$ (see \cite[Corollary 4.3]{kkw} and \cite[Remark 3.13]{llr2} for counterexamples). If we require the Daugavet property simultaneously on $X$ and on $Y$, there are promising results in \cite{mr21,rtv} that suggest that $X\pten Y$ should always enjoy the Daugavet property.

In spite of the previous results, it is conceivable that if $\mathcal F(M)$ has the Daugavet property then $\mathcal F(M)\pten X$ has the Daugavet property (it is explicitly asked in \cite[Question 1]{gpr18} and in \cite[Remark 3.8]{lr2020}). Let us mention, for instance, that it is known that if $\mathcal F(M)$ satisfies that every convex combination of slices of $B_{\mathcal F(M)}$ has diameter two, then $\mathcal F(M)\pten X$ also satisfies this property for every non-zero Banach space $X$ \cite[Corollary 3.5]{lr2020}, which contrasts with the fact that the above property is not inherited from just one of the factors by taking projective tensor products \cite[Theorem 3.8]{llr2}. On the other hand, in \cite[Proposition 3.11]{gpr18} it is proved that $\mathcal F(M)\pten X$ has the Daugavet property if $\mathcal F(M)$ has the Daugavet property and the pair $(M,X^*)$ satisfies that, for every Lipschitz function $f:N\subseteq M\longrightarrow X^*$, there is always a norm-preserving extension $F:M\longrightarrow X^*$ (i.e. $F_{|N}=f$ and $\Vert F\Vert=\Vert f\Vert$). As a natural substitute of classical McShane theorem, the above condition on the pair $(M,X^*)$ has been part of a number of results concerning the space $\Lip(M,X)$ (see e.g. \cite{blrohlip}, \cite[Proposition 3.11]{gpr18} and \cite[Theorem 2.6]{lr2020}).

The aim of this paper is to provide examples of metric spaces $M$ so that $\mathcal F(M)\pten X$ has the Daugavet property for every non-zero Banach space $X$. This is done through requiring on $M$ weak conditions of extensions of Lipschitz functions. There are two main theorems in this line. On the one hand, in Theorem \ref{theo:daugaF(M)finCEP} we prove that if $\mathcal F(M)$ has the Daugavet property and $M$ has the finite CEP then $\mathcal F(M)\pten X$ has the Daugavet property for every non-zero Banach space $X$. Moreover, in Theorem \ref{theo:daugaLipgeneral} we prove that if $\mathcal F(M)$ has the Daugavet property and $M$ satisfies the CEP then even $\Lip(M,X)=L(\mathcal F(M),X)$ has the Daugavet property. Even though this condition on $M$ seems to be very restrictive it turn out that both of the previous theorem find application to give a number of new examples $\mathcal F(M)\pten X$ enjoying the Daugavet property like $M$ being a Banach space whose dual is isometrically an $L_1$-space or $M$ being a convex subset of a Hilbert space (see Theorems \ref{theo:daugahilbert} and \ref{theo:daugaL1preduals}). Moreover, observe that Theorems \ref{theo:daugaF(M)finCEP} and \ref{theo:daugaLipgeneral} provide a partial positive answer to the above mentioned questions \cite[Question 1]{gpr18} and \cite[Remark 3.8]{lr2020}.

Let us now describe the content of the paper. Section \ref{section:CEPandrelated} makes an intensive study of the CEP and finite CEP condition (see Definition \ref{defi:CEP}) in order to provide examples of metric spaces with the (finite) CEP. It is known from \cite[Section 2.3]{beli} that if $X$ is a Banach space with the CEP then $X$ is a Hilbert space when $X$ is stricly convex and $X$ is an $L_1$-predual when $X$ is not strictly convex. On the one hand, we obtain in Proposition \ref{prop:carasubhilCEP} that a subset $M$ of a Hilbert space $H$ satisfies the CEP and $\mathcal F(M)$ has the Daugavet property if, and only if, $M$ is a convex subset of $H$. On the other hand, we make a study of those metric spaces which are finitely injective (see Definition \ref{defi:injective}), which are characterised in \cite{arpa1956} metrically. In Theorem \ref{theo:caracompainj} we prove that $M$ is finitely injective if, and only if, for every Lipschitz function $f:Y\longrightarrow M$ with $f(Y)$ compact there exists, for every metric space $X$ containing $Y$, a norm preserving extension $F:X\longrightarrow M$. Observe that this results improves simultaneously results from \cite{arpa1956} and \cite{eslo} (see Remark \ref{rema:mejoraespinola}). Moreover, as a consequence of Theorem \ref{theo:caracompainj} we derive that if $M$ is finitely injective then $M$ satisfies the finite CEP. As a consequence, we conclude in Corollary \ref{coro:finiteCEPBanach} a characterisation of those Banach spaces with the finite-CEP.

In Section \ref{section:Daugavet} we prove the main results about the Daugavet property $\mathcal F(M)\pten X$ (Theorem \ref{theo:daugaF(M)finCEP}) and in $\Lip(M,X)$ (Theorem \ref{theo:daugaLipgeneral}), obtaining the previously announced examples where the previous theorems can be applied. Moreover, we apply our techniques to face the problem whether $\mathcal F(M)\pten X$ has octahedral norm whenever $\mathcal F(M)$ has octahedral norm, posed in \cite[Question 3.2]{blrohlip}. We prove in Theorem \ref{theo:OHCEPtarget} that this is the case when the space $X^*$ has the finite CEP and $M$ has any cluster point.

\section{Notation and preliminary results}\label{section:preliminaries}

Throughout the paper we will only consider real Banach spaces. Given a Banach space $X$, we will denote the closed unit ball and the unit sphere of $X$ by $B_X$ and $S_X$ respectively. We will also denote by $X^*$ the topological dual of $X$. Given two Banach spaces $X$ and $Y$ denote by $L(X,Y)$ the space of linear bounded operators from $X$ into $Y$. 

By a \textit{slice} of the unit ball $B_X$ of a Banach space $X$ we will mean a set of the following form
\[ S(B_X,f,\alpha):=\{x\in B_X:f(x)>1-\varepsilon\}\]
where $f\in S_{X^*}$ and $\alpha>0$. If $X$ is a dual space, say $X=Y^*$, by a weak-star slice of $B_{X^*}$ we will mean a slice $S(B_X,y,\alpha)$ where $y\in Y$. %Notice that slices are non-empty relatively weakly open and convex subsets of $B_X$ whose complement is also convex. 

Given a metric space $M$ and a point $x\in M$, we will denote by $B(x,r)$ the closed unit ball centered at $x$ with radius $r$. 
Let $M$ be a metric space with a distinguished point $0 \in M$.
The couple $(M,0)$ is commonly called a \emph{pointed metric space}.
By an abuse of language we will say only ``let $M$ be a pointed metric space'' and similar sentences.
Given a pointed metric space $M$ and a Banach space $X$, we will denote by $\Lip(M,X)$ the Banach space of all $X$-valued Lipschitz functions on $M$ which vanish at $0$ under the standard Lipschitz norm
$$\Vert f\Vert:=\sup\left\{ \frac{\Vert f(x)-f(y)\Vert}{d(x,y)}\ :\ x,y\in M, x\neq y \right\} .$$
First of all, notice that we can consider every point of $M$ as an origin with no loss of generality, because the $\Lip$-spaces associated to different distinguised points are isometrically isomorphic; in particular, if $M$ is a Banach space we will consider as selected point the $0\in M$. Moreover, we will simply write $\Lip(M)$ instead of $\Lip(M,\mathbb R)$. Furthermore, in view of the previous notation, for a given Lipschitz function between two arbitrary metric spaces $\varphi:M\longrightarrow N$ we will write $\Vert\varphi\Vert$ to denote the best Lipschitz constant of $\varphi$, namely, $\Vert\varphi\Vert:=\sup_{x\neq y}\frac{d(\varphi(x),\varphi(y))}{d(x,y)}$.

We denote by $\delta$ the canonical isometric embedding of $M$ into $\Lip(M)^*$, which is given by $\langle f, \delta(x) \rangle =f(x)$ for $x \in M$ and $f \in \Lip(M)$. We denote by $\mathcal{F}(M)$ the norm-closed linear span of $\delta(M)$ in the dual space $\Lip(M)^*$, which is usually called the \textit{Lipschitz-free space over $M$}; for background on this, see the survey \cite{godesurv} and the book \cite{weaver} (where it receives the name of ``Arens-Eells space''). It is well known that $\mathcal{F}(M)$ is an isometric predual of the space $\Lip(M)$ \cite[p. 91]{godesurv}. We will write $\delta_x:=\delta(x)$ for $x\in M$. It is not difficult to prove that given a complete pointed metric space $M$, if we consider a dense subset $D\subseteq M$ containing the distinguised point then $\mathcal F(D)$ is isometrically isomorphic to $\mathcal F(M)$, as a consequence of the fact that every Lipschitz function $f:D\longrightarrow\mathbb R$ can be uniquely extended to $M$ without increasing its Lipschitz norm. Because of this fact, the metric spaces will be assumed to be complete with no loss of generality.

A fundamental result in the theory of Lipschitz-free spaces is that, roughly speaking, Lipschitz-free spaces linearise Lipschitz maps. In a more precise language, given a pointed metric space $M$, a Banach space $X$ and a Lipschitz map $f:M\longrightarrow X$ such that $f(0)=0$, there exists a bounded linear operator $T_f:\mathcal F(M)\longrightarrow X$ such that $\Vert T_f\Vert=\Vert f\Vert$ defined by
$$T_f(\delta_m):=f(m) \ , \quad m\in M.$$
Moreover, the mapping $f\longmapsto T_f$ is an onto linear isometry between $\Lip(M,X)$ and the space of bounded operators $L(\mathcal F(M),X)$.
This linearisation property makes Lipschitz-free spaces a precious magnifying glass to study Lipschitz maps between metric spaces, and for example it relates some well-known open problems in the Banach space theory to some open problems about Lipschitz-free spaces (see \cite{godesurv}).

We recall that
the {\it{projective tensor product of two Banach spaces $X$ and $Y$}}, denoted by
$X\pten Y$, is the completion of $X\otimes Y$
under the norm given by 
$$\Vert u\Vert:=\inf \left \{\sum_{i=1}^n \Vert x_i\Vert\Vert y_i\Vert\ /\ n\in\mathbb N, x_i\in X, y_i\in Y\ \forall i\in\{1,\ldots, n\}, u=\sum_{i=1}^n x_i\otimes y_i \right\},$$ for every $u\in X\otimes Y$. 

We recall that the space $L(X, Y^*)$ is linearly isometric to the topological dual of $X\pten Y$ by the action $T(x\otimes y):=T(x)(y)$. We refer the reader to \cite{ryan} for background on tensor product spaces.

We say that a Banach space $X$ is an \textit{$L_1$-predual} if $X^*$ is isometric to an $L_1(\mu)$ space for some measure $\mu$. We refer the reader to \cite[Theorem 6.1]{linds64} for multiple characterisations of $L_1$-predual spaces in terms of properties of intersection of closed balls and in terms of extension of compact operators.

Coming back to the Daugavet property, let us exhibit the following (well known) characterisation of the Daugavet property which will be used in the results of Section \ref{section:Daugavet}.

\begin{theorem}\label{theo:cadaDaugavet}
Let $X$ be a Banach space. The following assertions are equivalent:
\begin{enumerate}
\item $X$ has the Daugavet property.
\item For every $x\in S_X$, every slice $S$ of $B_X$ and every $\varepsilon>0$ there exists an element $y\in S$ so that $\Vert x+y\Vert>2-\varepsilon$.
\item For every $x^*\in S_{X^*}$, every $w^*$-slice $S$ of $B_{X^*}$ and every $\varepsilon>0$ there exists an element $y^*\in S$ so that $\Vert x^*+y^*\Vert>2-\varepsilon$.
\item For every $x\in S_X$, every non-empty relatively weakly open subset $W$ of $B_X$ and every $\varepsilon>0$ there exists an element $y\in W$ so that $\Vert x+y\Vert>2-\varepsilon$.
\item For every $x,y\in B_X$ and every $\varepsilon>0$ there exists a net $\{y_s\}_{s\in S}\subseteq (1+\varepsilon)B_X$ so that $\{y_s\}\rightarrow y$ weakly and $\Vert x+y_s\Vert\geq 2-\varepsilon$ for every $s\in S$.
\end{enumerate}
\end{theorem}

\begin{proof}
For the equivalence of (1)$\Leftrightarrow$(2)$\Leftrightarrow$(3)$\Leftrightarrow$(4) we refer the reader, for instance, to \cite{werner}.

To prove (4)$\Rightarrow$(5) pick $x,y\in B_X$. Given any weak open $U$ containing $y$ in $B_X$ there exists $y_U\in U$ so that $\Vert x+y_U\Vert\geq 2-\varepsilon$. Let $S$ be the set of those weak open neighbourhoods of $y$, which is directed under the order given by $U\leq V\Leftrightarrow V\subseteq U$. Then $\{y_U\}_{U\in S}$ satisfies the requirements.

Finally, for (5)$\Rightarrow$(2), pick $x\in S_X$, a slice $S=S(B_X,f,\alpha)$ of $B_X$ and $\varepsilon>0$, and let us find $y\in S$ so that $\Vert x+y\Vert>2-\varepsilon$. Take $\delta>0$ small enough so that $\frac{1-\delta}{1+\delta}>1-\alpha$ and $2(1-\delta)>2-\varepsilon$. Select $z\in B_X$ so that $f(z)>1-\delta$.

By assumptions there exists a net $\{z_s\}\subseteq (1+\delta)B_X$ which is weakly convergent to $z$ and so that $\Vert x+z_s\Vert\geq 2-\delta$ holds for every $s$. Since $f(z_s)\rightarrow f(z)>1-\delta$ find $s$ big enough so that $f(z_s)>1-\delta$. Consider $y:=\frac{z_s}{\Vert z_s\Vert}$ which belongs to $B_X$. Moreover
$$f(y)=\frac{f(z_s)}{\Vert z_s\Vert}>\frac{1-\delta}{1+\delta}>1-\alpha,$$
so $y\in S$. Moreover
$$\Vert x+y\Vert\geq \Vert x+y_s\Vert-\Vert y_s-y\Vert>2-\delta-\delta>2-\varepsilon,$$
getting (2) and finishing the proof.
\end{proof}

Let us finish explaining the characterisation of when a space $\Lip(M)$ enjoys the Daugavet property. We say that $M$ is \textit{length} if, for every $x\neq y\in M$, $d(x,y)$ equals the infimum of the length of all rectifiable curves joining them. If the infimum is actually a minimum, we say that $M$ is \textit{geodesic}.

If $M$ is complete, it is known that $M$ is geodesic if, and only if, for every $x\neq y\in M$ there exists $z\in M$ so that $d(x,z)=d(y,z)=\frac{d(x,y)}{2}$. Also, $M$ is length if, and only if, for every $x\neq y\in M$ and every $\varepsilon>0$ it follows that $B\left(x,\frac{(1+\varepsilon)d(x,y)}{2}\right)\cap B\left(y,\frac{(1+\varepsilon)d(x,y)}{2}\right)\neq \emptyset$. See \cite[Theorem 3.2]{gpr18}.

For a complete pointed metric space $M$, the following are equivalent:
\begin{enumerate}
\item $\Lip(M)$ has the Daugavet property.
\item $\mathcal F(M)$ has the Daugavet property.
\item $M$ is length.
\item $M$ has \textit{property (Z)}, that is, for every $x,y\in M$ with $x\neq y$ and every $\varepsilon>0$ there exists $z\in M\setminus\{x,y\}$ so that
$$d(x,z)+d(y,z)<d(x,y)+\varepsilon\min\{d(x,z),d(y,z)\}.$$
\item $B_{\mathcal F(M)}$ does not have strongly exposed points.
\end{enumerate}

See \cite{am,gpr18} for background.

\section{Examples of metric spaces with the finite-CEP}\label{section:CEPandrelated}

In this section we will take a closer look to the properties of extensions of Lipschitz functions in order to exhibit examples of metric spaces where the results of Section \ref{section:Daugavet} applies.

Let us start with formal required definitions.

\begin{definition}\label{defi:CEP}
Let $M$ be a metric space.
\begin{enumerate}
\item We say that $M$ has the \textit{contraction-extension property (CEP)} if for every Lipschitz function $f:N\longrightarrow M$, with $N\subseteq M$, there exists a norm-preserving extension to the whole $M$. 

\item We say that $M$ has the \textit{finite contraction-extension property (finite CEP)} if for every Lipschitz function $f:N\longrightarrow M$, with $N\subseteq M$, so that $f(N)$ is finite there exists a norm-preserving extension to the whole $M$. 
\end{enumerate}
\end{definition}

Observe that CEP was considered for Banach spaces in \cite[Section 2.3]{beli}. In this section we will prove that CEP is strictly stronger than the finite CEP. 

In any case, it turns out that the finite CEP is a strong requirement on the metric space $M$. We can illustrate this just focusing on the case of Banach spaces. According to \cite[Section 2.3]{beli}, it follows that if $X$ is a Banach space, if $X$ has the CEP there are two possibilities: if $X$ is strictly convex then $X$ is a Hilbert space, otherwise $X$ is an $L_1$-predual. Actually, the remark after the proof of \cite[Theorem 2.11]{beli} shows that the previous result holds if we replace the CEP with the finite-CEP.

In the case of Hilbert spaces the converse also holds, that is, if $H$ is a Hilbert space then $H$ has the CEP (c.f. e.g. \cite[Theorem 1.12]{beli}).

In view of the previous result, we will examine which metric spaces $M$ contained in a Hilbert spaces which are geodesic and have the CEP. The following proposition is probably well known for speciallists, but let us include a short proof for the sake of completeness.

\begin{proposition}\label{prop:carasubhilCEP}
Let $M$ be a subset of a Hilbert space $H$. The following are equivalent:
\begin{enumerate}
\item $M$ is geodesic and has the CEP.
\item $M$ is geodesic and has the finite CEP.
\item $M$ is geodesic.
\item $M$ is length.
\item $M$ is convex.
\end{enumerate}
\end{proposition}

\begin{proof}
Observe that implications (1)$\Rightarrow$(2)$\Rightarrow$(3)$\Rightarrow$(4) are clear from the very definition.

Since $M$ is a subset of $H$, $M$ is convex if $M$ is length (see e.g. \cite[Theorem 2.9]{ikw}). This proves (4)$\Rightarrow$(5).

To prove (5)$\Rightarrow$(1), it is immediate that if $M$ is convex then $M$ is geodesic, so let us prove that $M$ has the CEP. Since $M$ is convex classical theory of Hilbert spaces implies the existence of a norm-one Lipschitz retraction $r:H\longrightarrow C$ which is consequence of the projection theorem. It is straightforward to conclude the CEP on $M$ from the CEP on $H$ and the retraction $r$. Indeed, given a subset $X\subseteq M$ and a Lipschitz function $f:X\longrightarrow M$, consider $i:M\longrightarrow H$ the inclusion map. Now $i\circ f:X\longrightarrow H$ is a $\Vert f\Vert$-Lipschitz function. Since $H$ has the CEP there exists an extension $F:H\longrightarrow H$ so that $\Vert F\Vert=\Vert f\Vert$. It remains to consider $r\circ F$ restricted to $M$.
\end{proof}

The situation for $L_1$-preduals is a bit more delicated. In general, there are $L_1$-predual spaces which do not enjoy the CEP (see e.g. \cite[P.59]{beli}, where it is established that there is no separable $L_1$-predual with the CEP). Our aim is to show that, for the finite CEP, the situation is different. In order to do so, let us introduce a bit of notation.

\begin{definition}\label{defi:injective}
Let $M$ be a metric space. We say that $M$ is:
\begin{enumerate}
\item \textit{injective} if for every pair of metric spaces $Y\subseteq X$ and every Lipschitz function $f:Y\longrightarrow M$ there exists a norm-preserving Lipschitz extension, that is, there exists $F:X\longrightarrow M$ so that $\Vert F\Vert=\Vert f\Vert$ and $F_{|Y}=f$.

\item \textit{finitely injective} if for every pair of finite metric spaces $Y\subseteq X$ and every Lipschitz function $f:Y\longrightarrow M$ there exists a norm-preserving Lipschitz extension, that is, there exists $F:X\longrightarrow M$ so that $\Vert F\Vert=\Vert f\Vert$ and $F_{|Y}=f$.
\end{enumerate}
\end{definition}

Injective metric spaces have been widely studied in the literature and include, for instance, the injective Banach spaces. We refer the reader to \cite[Chapter 1 and 2]{beli}, \cite{despa} and references therein for background. Clearly, injective metric spaces enjoy the CEP, but the converse is not true, for instance, by Proposition \ref{prop:carasubhilCEP}.

Observe that, from \cite{arpa1956}, a metric space $M$ is finitely injective if, and only if, $M$ is geodesic and enjoys the following property of intersections of balls: if $\{B(x_i,r_i): 1\leq i\leq n\}$ is a finite family of closed balls so that $B(x_i,r_i)\cap B(x_j,r_j)\neq \emptyset$ if $i\neq j$ then $\bigcap\limits_{i=1}^n B(x_i,r_i)\neq \emptyset$ (the previous condition in \cite{arpa1956} is $M$ being \textit{$m$-hyperconvex} for every $m\in\mathbb N$). Thanks to this characterisation, it is easy to see that $L_1$-predual Banach spaces are finitely injective \cite[Theorem 6.1]{linds64}.

Observe also that, with the previous definition, it is not clear that finite-injectivity implies finite-CEP. This is because the definition of finite injectivity allows extending Lipschitz functions where the domains are finite metric spaces. On the other hand, the definition of finite-CEP requires extending Lipschitz functions whose range is finite but the domain spaces can be infinite. Because of that, we will prove that finite injectivity actually allows extending Lipschitz functions with finite range.

For this, we start with the following lemma, whose proof can be found in \cite[Theorem 4.5 and the remark after]{linds64}.

\begin{lemma}\label{lemma:compip}
Let $M$ be a complete metric space. If $M$ has the finite intersection property, then $M$ has the \textit{compact intersection property}, i.e. if $\{B(x_i,r_i): i\in I\}$ is a family of pairwise intersecting balls and $\{x_i: i\in I\}$ is relatively compact in $M$, then $\bigcap\limits_{i\in I} B(x_i,r_i)\neq \emptyset$.
\end{lemma}

Now we have the following theorem.

\begin{theorem}\label{theo:caracompainj}
Let $M$ be a complete metric space. The following are equivalent:
\begin{enumerate}
\item $M$ is finitely injective.
\item $M$ is geodesic and has the compact intersection property.
\item For every pair of compact metric spaces $Y\subseteq X$ and every Lipschitz function $f:Y\longrightarrow M$ there exists a norm-preserving extension $F:X\longrightarrow M$. 
\end{enumerate}
\end{theorem}

\begin{proof}
(1)$\Rightarrow$(2) follows from the results of \cite{arpa1956} and by Lemma \ref{lemma:compip} and (3)$\Rightarrow$(1) is clear.

Let us now prove (2)$\Rightarrow$(3). Take $f$ as in the hypothesis and define $\mathfrak{F}$ to be the set of those pairs $(Z,g)$ so that $Z$ is a closed subset of $X$ containing $Y$ and $g:Z\longrightarrow M$ is a norm-preserving Lipschitz extension of $f$. Observe that $(Y,f)\in\mathfrak F$. Moreover, declare the usual order $(Z_1,g_1)\leq (Z_2,g_2)$ iff $Z_1\subseteq Z_2$ and $g_{2|Z_1}=g_1$. It is not difficult to prove that $(\mathfrak{F},\leq)$ is inductive.

By Zorn lemma take a maximal element $(Z,F)$. We claim that $Z=X$. Assume by contradiction that $Z\neq X$ and take $x\in X\setminus Z$. In order to define an extension of $F$ to $x$, note that the family $\{B(F(y),\Vert f\Vert d(y,x)): y\in Z\}$ in $M$ is a family of pairwise intersecting balls (this is because given two of these balls the distance between centers is smaller than the sum of radii, a condition which guarantees that the intersection of balls is non-empty in geodesic metric spaces) whose centers are a compact set (since $Z$ is compact as being closed in $X$). By assumption we get that there exists $m_0\in\bigcap\limits_{y\in Z} B(F(y), \Vert f\Vert d(y,x))$. Now defining $\hat F(x)=m_0$ we get a norm preserving extension of $F$ to $Z\cup\{x\}$. Indeed, given $y\in Z$ we get that $\hat F(x)=m_0\in B(f(y),\Vert f\Vert d(y,x))$, so $d(\hat F(x),\hat F(y))=d(m_0,F(y))\leq \Vert f\Vert d(x,y)=\Vert F\Vert d(x,y)$, which implies that $\Vert \hat F\Vert=\Vert f\Vert$. Consequently, $(Z\cup\{x\},\hat F)\in\mathfrak F$ and $(Z,F)<(Z\cup\{x\},\hat F)$, which contradicts the maximality of $(Z,F)$ and finishes the proof.
\end{proof}

For the following theorem we need to introduce a bit of notation. Given a metric space $M$, the \textit{injective envelope of $M$} is a pair $(\varepsilon M,i)$, where $\varepsilon M$ is an injective metric space and $i:M\longrightarrow \varepsilon M$ is an isometry and no proper subspace of $\varepsilon M$ contains $i(M)$.

Note that in \cite{isbell} it is proved that every metric space has an injective envelope which is, up to onto isometry, unique. Moreover, it is proved that the injective envelope of a compact metric space is also compact \cite[p.73]{isbell}.

Now we are ready to prove the following result.

\begin{theorem}\label{theo:carafinincompextension}
Let $M$ be a finitely injective metric space. Then, for every pair of metric spaces $Y\subseteq X$ and every Lipschitz mapping $f:Y\longrightarrow M$ so that $f(Y)$ is compact there exists a norm-preserving extension $F:X\longrightarrow M$ so that $F(X)$ is compact too.
\end{theorem}

\begin{proof}
Let us assume with no loss of generality that $\Vert f\Vert=1$. Let us consider the injective envelop of the metric space $f(N)$, which is compact, and call $e:f(N)\longrightarrow \varepsilon f(Y)$ the canonical embedding. Call also $i:f(Y)\longrightarrow M$ the canonical inclusion. Now set $e\circ f:Y\longrightarrow \varepsilon f(Y)$, which is a norm-one Lipschitz function. Since $\varepsilon f(Y)$ is injective there exists a norm-preserving extension $G:X\longrightarrow \varepsilon f(Y)$. Observe that $G$ satisfies that $G=e\circ f$ on $Y$.

Now set $p:e(f(Y))\longrightarrow M$ so acting as $p(e(f(y)))=f(y)$ (i.e. it is $e^{-1}$ on $e(f(Y))$), which is a norm-one Lipschitz function. Moreover, $e(f(Y))$ is compact, and $\varepsilon f(Y)$ is compact too since $f(Y)$ is compact. By Theorem \ref{theo:caracompainj} there exists a norm-preserving extension $P:\varepsilon f(Y)\longrightarrow M$. Now set $F:=P\circ G:X\longrightarrow M$, which is a norm-one Lipschitz function. Moreover, given $y\in Y$, we get
$$P(G(y))=P(e(f(y)))=p(e(f(y)))=f(y),$$
so $F$ extends $f$ and the proof is finished.
\end{proof}

Let us obtain consequences of the above result.

\begin{remark}\label{rema:mejoraespinola}
\begin{enumerate}
\item From \cite[Theorem 2.4]{eslo} it is obtained that if $M$ is finitely injective then, given any pair of metric spaces $Y\subseteq X$ and every $f:Y\longrightarrow M$ so that $f(Y)$ is compact there exists, for every $\varepsilon>0$, a Lipschitz extension $F:X\longrightarrow M$ so that $F(X)$ is compact and so that $\Vert F\Vert\leq (1+\varepsilon)\Vert f\Vert$. Observe that Theorem \ref{theo:carafinincompextension} allows to get $\varepsilon=0$ in the above theorem.

\item If $X$ is a Banach space, then Theorem \ref{theo:carafinincompextension} should be compared with \cite[Theorem 3.5]{rueda21}. In that result, it is proved that $X$ is an $L_1$-predual if, and only if, for every Lipschitz function $f:N\longrightarrow X$ which is \textit{Lipschitz compact} (i.e. the set $\{\frac{f(x)-f(y)}{d(x,y)}: x\neq y\}\subseteq X$ is relatively compact) there exists, for every $\varepsilon>0$ and every $M\supset N$, an extension $F:M\longrightarrow X$ with $\Vert F\Vert\leq (1+\varepsilon)\Vert f\Vert$.
\end{enumerate}
\end{remark}

As a consequence of Theorem \ref{theo:carafinincompextension} we obtain that finitely injective metric spaces enjoy the finite-CEP. 

\begin{corollary}\label{coro:finiteCEP}
Let $M$ be a metric space. If $M$ is finitely injective then $M$ has the finite-CEP.
\end{corollary}

\begin{proof}
Given $N\subseteq M$ and a Lipschitz function $f:N\longrightarrow M$ so that $f(N)$ is finite then $f(N)$ is in particular compact, and now a norm-preserving extension $F:M\longrightarrow  M$ is guaranteed by Theorem \ref{theo:carafinincompextension}.
\end{proof}

In the case of Banach space, a complete characterisation is obtained from the fact that $L_1$-predual spaces are finitely injective.

\begin{corollary}\label{coro:finiteCEPBanach}
Let $X$ be a Banach space. 
\begin{enumerate}
\item If $X$ is strictly convex, then $X$ has the finite-CEP if, and only if, $X$ is a Hilbert space and if, and only if, $X$ has the CEP.
\item If $X$ is not strictly convex, then $X$ has the finite-CEP if, and only if, $X$ is an $L_1$-predual.
\end{enumerate}
\end{corollary}

\begin{proof}
If $X$ is stricly convex, if $X$ has the finite-CEP the comment after the proof of \cite[Theorem 2.11]{beli} yields that $X$ is a Hilbert space. For the converse, if $X$ is a Hilbert space then $X$ has the CEP by \cite[Theorem 1.12]{beli}.

If $X$ is not strictly convex, if $X$ has the finite-CEP then again the comment after the proof of \cite[Theorem 2.11]{beli} yields that $X$ is an $L_1$-predual. For the converse, \cite[Theorem 6.1, (13)]{linds64} implies that $X$ is finitely injective by Theorem \ref{theo:caracompainj}. Now Corollary \ref{coro:finiteCEP} implies that $X$ has the finite-CEP, as desired.
\end{proof}

\begin{remark}
Observe that every separable $L_1$-predual space $X$ is an example of metric space with the finite-CEP but failing the CEP. To the best of our knownledge, there is not a characterisation of those $L_1$-predual spaces with the CEP. It is conjectured in \cite[P. 44]{beli} that this class can coincide precisely with the injective Banach spaces.
\end{remark}

Let us now give a new characterisation of finitely injective metric spaces inspired by \cite[Theorem 6.1 (11)]{linds64}. In that result, it is proved that a Banach space $X$ is an $L_1$-predual if, and only if, $X$ has the MAP and for every compact operator $T:X\longrightarrow X$ there exists, for every $Z\supset X$ with $\operatorname{dim}(Z/X)=1$, a compact extension $\hat T:Z\longrightarrow X$ with $\Vert \hat T\Vert=\Vert T\Vert$. In such spirit, we have the following analogue for finitely injective metric spaces.

\begin{proposition}
Let $M$ be a complete metric space. Then $M$ is finitely injective if, and only if, the following two assertions hold:
\begin{enumerate}
\item For every finite subset $F\subseteq M$ there exists a norm-one Lipschitz function $f:M\longrightarrow M$ so that $f(M)$ is compact and $f(x)=x$ holds for every $x\in F$.
\item For every Lipschitz function $g:M\longrightarrow M$ so that $g(M)$ is compact there exists, for every $X\supset M$ so that $X\setminus M$ is a singleton, a norm-preserving extension $G:X\longrightarrow M$.
\end{enumerate}
\end{proposition} 

\begin{proof}
It is immediate from Theorem \ref{theo:carafinincompextension} that if $M$ is finitely injective then it satisfies (1) and (2). For the converse, assume that $M$ satisfies (1) and (2) and let us prove that it is finitely injective. In order to do so, let us prove that $M$ is geodesic and satisfies the finite intersection property of balls in virtue of the results of \cite[Section 2]{arpa1956}. We consider $M$ as a subset of $\ell_\infty(I)$ for certain subset $I$ (observe that $M$ can be embedded isometrically in an $\ell_\infty(I)$ space, for instance considering the isometry $\delta:M\hookrightarrow \mathcal F(M)$ and now by the well known fact that any Banach space embeds isometrically in an $\ell_\infty(I)$).

\textit{$M$ is geodesic}: Let $x\neq y\in M$ and let us  prove that $B\left(x,\frac{d(x,y)}{2}\right)\cap B\left(y,\frac{d(x,y)}{2}\right)\neq \emptyset$. By (1) find a norm-one Lipschitz map $f:M\longrightarrow M$ so that $f(M)$ is compact and $f(x)=x, f(y)=y$. Find $z\in \ell_\infty(I)$ so that $z\in B_{\ell_\infty(I)}\left(x,\frac{d(x,y)}{2}\right)\cap B_{\ell_\infty(I)}\left(y,\frac{d(x,y)}{2}\right)$. By (2) there exists a norm-one preserving extension $F:M\cup\{z\}\longrightarrow M$. Observe that $F(z)\in B\left(x,\frac{d(x,y)}{2}\right)\cap B\left(y,\frac{d(x,y)}{2}\right)$. This proves that $M$ is geodesic by completeness.
\vspace{0.3cm}

\textit{$M$ has the finite intersection property}: Let $B(x_i,r_i)$, $1\leq i\leq n$ be a family of closed balls which is pairwise intersecting in $M$. Let $z\in \ell_\infty(I)$ be an element in $\bigcap\limits_{i=1}^n B_{\ell_\infty(I)}(x_i,r_i)$. By (1) take a norm-one Lipschitz function $f:M\longrightarrow M$ so that $f(M)$ is compact and $f(x_i)=x_i$. By (2) there exists an extension $F:M\cup \{z\}\longrightarrow M$. It follows that $F(z)\in \bigcap\limits_{i=1}^n B(x_i,r_i)$, and we are done.
\end{proof}

\begin{remark}
Observe that condition (1) is kind of a metric version of the metric approximation property. On the other hand, observe that we can not replace ``$f(M)$ is compact'' with ``$f(M)$ is finite''. This follows because if $M$ is finitely injective then it is path-connected as being geodesic. Consequently, every Lipschitz mapping $f:M\longrightarrow M$ taking finitely many values is constant.
\end{remark}

In order to exhibit examples, apart from the context of Banach spaces, where being finitely injective is the same as being an $L_1$-predual, let us point out that in \cite{despa} a characterisation of those subsets of $\ell_\infty(I)$ which are injective is given. Since every metric space $M$ embeds isometrically in an $\ell_\infty(I)$ set for a suitable $I$, this gives a theoretical characterisation of all the injective metric spaces. In order to give particular examples of finitely injective metric spaces, let us consider the following example.

\begin{example}
Let $M$ be a complete finitely injective metric space. Let $B(z_i,R_i), i\in I$ be a family of balls which are pairwise intersecting and so that the set $\{z_i: i\in I\}\subseteq M$ is compact. Then $X:=\bigcap\limits_{i\in I}B(z_i,R_i)$ is finitely injective. In particular, closed balls in a finitely injective metric space are finitely injective.
\end{example}

\begin{proof}
We have to prove that $X$ is geodesic and that it satisfies the finite intersection of properties of balls. It is immediate that if $X$ reduces to a singleton then both properties are clearly satisfied.

On the other hand, let us assume that it has more than one point. Let us start by proving that $X$ is geodesic. To this end since $X$ is complete as being a closed subset of the complete space $M$, take $x\neq y\in X$ and let us prove that $B_X\left(x,\frac{d(x,y)}{2}\right)\cap B_X\left(y,\frac{d(x,y)}{2}\right)\neq \emptyset$. In order to do so, consider the family of balls $\{B(z_i,R_i):i\in I\}\cup\left\{B\left(z,\frac{d(x,y)}{2}\right): z\in\{x,y\}\right\}$. This family of balls is pairwise intersecting ($B\left(x,\frac{d(x,y)}{2}\right)\cap B\left(y,\frac{d(x,y)}{2}\right)\neq \emptyset$ since $M$ is geodesic as being finitely injective). Moreover, the set $\{x_i:i\in I\}\cup\{x,y\}$ is compact because it is the union of two compact subsets of $M$. By Lemma \ref{lemma:compip} we derive that $B\left(x,\frac{d(x,y)}{2}\right)\cap B\left(y,\frac{d(x,y)}{2}\right)\cap\bigcap\limits_{i\in I} B(z_i,R_i)=B\left(x,\frac{d(x,y)}{2}\right)\cap B\left(y,\frac{d(x,y)}{2}\right)\cap X\neq \emptyset$. But this precisely means that $B_X\left(x,\frac{d(x,y)}{2}\right)\cap B_X\left(y,\frac{d(x,y)}{2}\right)\neq \emptyset$. Summarising, we have proved that every pair of points in the complete metric space $X$ has a medium point. Consequently, $X$ is geodesic.

In order to prove that $X$ has the finite intersection property, take $B_X(x_j,r_j), 1\leq j\leq k$ be a family of pairwise intersecting balls in $X$. Let us prove that $\emptyset\neq \bigcap\limits_{j=1}^k B_X(x_j,r_j)=\bigcap\limits_{j=1}^k B(x_j,r_j)\cap \bigcap\limits_{i\in I} B(z_i,R_i)=\bigcap\limits_{j=1}^k B(x_j,r_j)\cap \bigcap \limits_{i\in I} B(z_i,R_i)$. But this is consequence of Lemma \ref{lemma:compip} since the family $\{B(z_i,R_i):i\in I\}\cup\{B(x_j,r_j): 1\leq j\leq k\}$ is a family of pairwise intersecting balls whose set of centers $\{z_i:i\in I\}\cup\{x_1,\ldots x_k\}$ is compact.
\end{proof}

\section{Spaces of Lipschitz functions and the Daugavet property}\label{section:Daugavet}

In this section we will exhibit how the finite CEP and the CEP allows to obtain consequences about the Daugavet property in spaces of vector-valued Lispchitz functions. Let us start with the first theorem of the section.

\begin{theorem}\label{theo:daugaF(M)finCEP}
Let $M$ be a geodesic pointed metric space with the finite-CEP. If $X$ is a non-zero Banach space, then $\mathcal F(M)\pten X$ has the Daugavet property.
\end{theorem}

The proof is very involved, so we will split the construction in a number of lemmata of independent interest.

\begin{lemma}\label{lemma:fininyextidentidad}
Let $M$ be a metric space with the finite-CEP. Let $m_1,\ldots, m_p,x_0$ be different points in $M$ and $\eta>0$. There exists a Lipschitz function $\varphi:M\longrightarrow M$ with $\Vert \varphi\Vert\leq 1+\eta$ so that $\varphi(m_i)=m_i$ for $1\leq i\leq p$ and $\varphi=x_0$ on $B(x_0,r)$, for suitable $r>0$ small enough.
\end{lemma}

\begin{proof} Define $R:=\min\{d(x_0,m_i): 1\leq i\leq p\}$, take $0<r<R$ so that $\frac{r}{R-r}<\eta$ and define $\varphi:B(x_0,r)\cup\{m_{i}: 1\leq i\leq p\}\longrightarrow M$ by $\varphi=x_0$ on $B(x_0,r)$ and $\varphi(m_i)=m_{i}$ for $1\leq i\leq m$. Let us estimate the norm of $\varphi$. In order to do so,  and take $x\neq y\in B(x_0,r)\cup\{m_{i}: 1\leq i\leq p\}$, and let us estimate $\frac{d(\varphi(x),\varphi(y))}{d(x,y)}$. The non-trivial case is when $x\in B(x_0,r)$ and $y\in\{ m_1,\ldots, m_p\}$. In such case
$$\frac{d(\varphi(x),\varphi(y))}{d(x,y)}=\frac{d(y,x_0)}{d(x,y)}\leq \frac{d(x,y)+d(x,x_0)}{d(x,y)}=1+\frac{d(x,x_0)}{d(x,y)}\leq 1+\frac{r}{R-r}<1+\eta$$
 Now, since $M$ has the finite CEP and $\varphi$ is a Lipschitz function of finite image, $\varphi$ can be extended in a norm-preserving way (still denoted by $\varphi$) to $\varphi:M\longrightarrow M$, which satisfies all our requirements.
\end{proof}

In the previous lemma we have found, roughly speaking, Lipschitz functions which are constant on a ball and the identity at finitely many points out of this ball. In the following lemma we make a kind of reverse process, we find Lipschitz functions which are constant out of a ball and acts as the identity at some points of the ball.

\begin{lemma}\label{lemma:finiteceplemalocalF(M)dau}
Let $M$ be a metric space with the finite CEP. Then, for every $x_0\in M$, every $\delta>0$ and every $x,y\in M, x\neq y$ so that $\max\left\{\frac{d(x_0,x)}{\delta-d(x_0,x)},\frac{d(x_0,y)}{\delta-d(x_0,y)}\right\}<1$ there exists a Lipschitz function $\psi:M\longrightarrow M$ so that $\Vert\psi\Vert\leq 1$, $\psi=x_0$ on $M\setminus B(x_0,\delta)$, $\psi(x_0)=x_0$, $\psi(x)=x$ and $\psi(y)=y$.
\end{lemma}

\begin{proof}
Given $\delta>0$, $x_0,x,y$ as in the hypothesis define $\psi: \{0,x,y\}\cup M\setminus B(0,\delta)\longrightarrow M$ by $\psi(x)=x$, $\psi(y)=y$ and $\psi=x_0$ otherwise. It is not difficult to prove that $\psi$ is $1$-Lipschitz. Since $M$ has the finite CEP and the image of $\psi$ is finite, we can extend it in a norm preserving way to a function (still denoted by $\psi$) defined on the whole $M$, which satisfies all our requirements. \end{proof}

We end with a lemma which estimates the norm of the sum of two functions when the sets where each function is not constant are very separated.

\begin{lemma}\label{lemma:compuortolip}
Let $M$ be a geodesic metric space and $X$ be a Banach space. Let $f,g:M\longrightarrow X$ be two  Lipschitz functions with $\Vert f\Vert\leq 1$ and $\Vert g\Vert\leq 1$. Assume that there exists $m\in M$ and $0<\delta<R$ so that
\begin{enumerate}
\item $g$ is constant on $B(m,R)$.
\item $f(x)=f(m)$ holds for every $x\in M\setminus B(m,\delta)$.
\end{enumerate}
Then $\Vert f+g\Vert\leq 1+\frac{\delta}{R-\delta}.$
\end{lemma}

\begin{proof}
Let $x, y\in M$ with $x\neq y$. Let us estimate $A:=\frac{\Vert (f(x)+g(x))-(f(y)+g(y))\Vert}{d(x,y)}=\frac{\Vert f(x)-f(y)+g(x)-g(y)\Vert}{d(x,y)}$. Let us observe that if $f(x)-f(y)=0$ or $g(x)-g(y)=0$ then $A\leq 1$.

The unique possibility for the previous condition not to hold is that, up to relabeling, $x\notin B(m,R)$ and $y\in B(m,\delta)$. In that case $f(x)=f(m)$ and $g(y)=g(m)$. Consequently
$$A\leq \frac{\Vert f(y)-f(m)\Vert+\Vert g(x)-g(m)\Vert}{d(x,y)}\leq \frac{d(y,m)+\Vert g(x)-g(m)\Vert}{d(x,y)}$$
As $M$ is geodesic we can find a point $z\in M$ so that $d(z,m)=R$ and $d(x,m)=d(x,z)+R$ (take an isometric curve $\alpha:[0,d(x,m)]\longrightarrow M$ so that $\alpha(0)=m$ and $\alpha(d(x,m))=x$, and $z=\alpha(R)$ satisfies the requirements). Since $z\in B(m,R)$ we derive that $g(m)=g(z)$ because $g$ is constant on $B(m,R)$, so $\Vert g(x)-g(m)\Vert=\Vert g(x)-g(z)\Vert\leq d(x,z)=d(x,m)-R$. Consequently
$$A\leq \frac{d(x,m)-R+d(y,m)}{d(x,y)}\leq \frac{d(x,m)-R+d(y,m)}{d(x,m)-d(y,m)}\leq \frac{d(x,m)-R+d(y,m)}{d(x,m)-\delta}$$
Since $R>\delta$ we get $\frac{d(x,m)-R}{d(x,m)-\delta}\leq 1$, hence
$$A\leq 1+\frac{d(y,m)}{d(x,m)-\delta}\leq 1+\frac{\delta}{R-\delta},$$
and the lemma is proved by the abritrariness of $x,y\in M, x\neq y$.
\end{proof}

Now we are ready to show the pending proof.

\begin{proof}[Proof of Theorem \ref{theo:daugaF(M)finCEP}]
In view of the equality $(\mathcal F(M)\pten X)^*=L(\mathcal F(M),X^*)=\Lip(M,X^*)$, it is enough to prove that, given $f\in S_{\Lip(M,X^*)}$, a weak-star slice $S=S(B_{\Lip(M,X^*)},z,\alpha)$ and let us prove that there are elements $g\in S$ so that $\Vert f+g\Vert$ is a close to $2$ as we wish. In order to do so, pick $\varepsilon>0$.

We can assume, by a density argument (observe that $\{\delta_t:t\in M\}$ has a dense linear span in $\mathcal F(M)$ and $\{\mu\otimes x: \mu\in \mathcal F(M), x\in X\}$ has a dense span in $\mathcal F(M)\pten X$), that $z=\sum_{i=1}^p \delta_{m_i}\otimes x_i$ for suitable $p\in\mathbb N, m_1,\ldots, m_p\in M$ and $x_1,\ldots, x_p\in X$. Take $\eta_0>0$ and $h\in S$ with $h(z)=\sum_{i=1}^p h(m_i)(x_i)>(1+\eta_0)(1-\alpha)$. We can assume with no loss of generality that $(1+\eta_0)(1-\varepsilon)<1$.

Since $M$ is geodesic then there are infinitely many points $a_n$ with the following property: for every $\delta>0$ it follows $\Vert f_{|B(a_n,\delta)}\Vert>(1+\eta_0)(1-\varepsilon)$. Indeed, take $x\in S_X$ so that $\Vert x\circ f\Vert>(1+\eta_0)(1-\varepsilon)$, where $x\circ f:M\longrightarrow \mathbb R$ is defined by $(x\circ f)(t):=f(t)(x)$. Since $M$ is geodesic, the existence of the points $a_n$ is guaranteed by \cite[Proposition 2.3]{ikw}.

 Take one such point (say $x_0$) with so that $x_0\notin\{m_1,\ldots, m_p,0\}$.  Let $R:=\min\{d(x_0,m_1),\ldots, d(x_0,m_p), d(x_0,0)\}$. Apply  Lemma \ref{lemma:fininyextidentidad} for $0<\eta<\frac{\eta_0}{2}$  and find $\varphi:M\longrightarrow M$ and $r>0$ satisfying $\Vert \varphi\Vert\leq 1+\eta$, $\varphi(z)=z$ for $z\in \{m_1,\ldots, m_p, 0\}$ and $\varphi=x_0$ on $B(x_0,r)$. Now take $0<\delta<r$ and find $x,y\in B(0,\delta)$ with $x\neq y$, so that $\max\{\frac{d(x_0,x)}{\delta-d(x_0,x)}, \frac{d(x_0,y)}{\delta-d(x_0,y)}\}<1$ and such that $\frac{\Vert f(x)-f(y)\Vert}{d(x,y)}>(1+\eta_0)(1-\varepsilon)$. Observe that such points $x,y$ exist because of the assumption on $x_0$. Find, by Lemma \ref{lemma:finiteceplemalocalF(M)dau}, a norm-one Lipschitz function $\psi:M\longrightarrow M$ so that $\psi(z)=x_0$ for every $z\in M\setminus B(x_0,\delta)$, $\psi(x_0)=x_0, \psi(x)=x$ and $\psi(y)=y$.

Consider $g:=h\circ\varphi+(f\circ\psi-(f\circ\psi)(0)):M\longrightarrow X^*$. Observe that $g(x)=h(x_0)+f(x)-(f\circ\psi)(0)$ and $g(y)=h(x_0)+f(y)-(f\circ\psi)(0)$. Consequently $g(x)-g(y)=f(x)-f(y)$ and then $\Vert f+g\Vert\geq 2\frac{\Vert f(x)-f(y)\Vert}{d(x,y)}>2(1+\eta_0)(1-\varepsilon)$. On the other hand, it is clear that $f\circ\psi-(f\circ\psi)(0)$ equals $0$ in $M\setminus B(x_0,\delta)$, so $g(z)=\sum_{i=1}^p h(\varphi(m_i))(x_i)=\sum_{i=1}^p h(m_i)(x_i)>(1+\eta_0)(1-\alpha)$ since $m_i\notin B(0,\delta)$. We claim that $\Vert g\Vert\leq 1+\eta+\frac{\delta}{r-\delta}$. Indeed, it remains to apply Lemma \ref{lemma:compuortolip} to $\frac{h\circ\varphi}{1+\eta}$ and $(f\circ\psi-(f\circ\psi)(0))$ and making an easy perturbation argument.

Now we can select $\eta$ and $\delta$ so that $1+\eta+\frac{\delta}{r-\delta}<1+\eta_0$. Now $\frac{g}{1+\eta_0}\in B_{\Lip(M,X^*)}$ satisfies that
$$\frac{g}{(1+\eta_0)}(z)>1-\alpha,$$
so $\frac{g}{1+\eta_0}\in S$. Moreover,
\[
\begin{split}
\left\Vert f+\frac{g}{1+\eta_0} \right\Vert\geq \Vert f+g\Vert-\left\Vert g-\frac{g}{1+\eta_0}\right\Vert>2(1+\eta_0)(1-\varepsilon)-\eta_0.
\end{split}
\]
Since $\varepsilon$ was arbitrary and $\eta_0>0$ can be taken as small as desired we conclude the desired result.
\end{proof}

Let us now pass to the second main theorem of the section. Under the stronger condition that $M$ actually has the CEP, we get the following result.

\begin{theorem}\label{theo:daugaLipgeneral}
Let $M$ be a pointed geodesic metric space with the CEP. Then $L(\mathcal F(M),X)=\Lip(M,X)$ has the Daugavet property for every Banach space $X$.
\end{theorem}

This time we need just a preliminary lemma, which is a strengthening of Lemma \ref{lemma:fininyextidentidad}.

\begin{lemma}\label{lema:daugaextindenlocons}
Let $M$ be a metric space with the CEP. For every $\varepsilon>0$, every $m\in M$ and every $R>0$ there exists $\delta>0$ and a Lipschitz function $f:M\longrightarrow M$ with $\Vert f\Vert\leq 1+\varepsilon$, $f(x)=x$ for every $x\notin B(x,R)$ and $f(x)=m$ on $B(m,\delta)$.
\end{lemma}

\begin{proof}
Take $0<\delta<R$, set $A:=B(m,\delta)\cup M\setminus B(m,R)$ and define $f:A\longrightarrow M$ by $f(x)=x$ if $x\notin B(m,R)$ and $f(x)=m$ if $x\in B(m,\delta)$. Let us estimate the norm of $f$. Take $x\neq y\in A$. It is immediate that $\frac{d(f(x),f(y))}{d(x,y)}\leq 1$ if either $\{x,y\}\subseteq B(m,\delta)$ or $\{x,y\}\subseteq M\setminus B(m,R)$. So let us examine the case that $y\in B(m,\delta)$ and $x\notin B(m,R)$. In such case we have
\[\begin{split}
\frac{d(f(x),f(y))}{d(x,y)}=\frac{d(x,m)}{d(x,y)}\leq \frac{d(x,y)+d(y,m)}{d(x,y)}\leq 1+\frac{d(y,m)}{d(x,m)-d(y,m)}\leq 1+\frac{\delta}{R-\delta}.
\end{split}\]
Taking supremum on $x\neq y$ we get $\Vert f\Vert\leq 1+\frac{\delta}{R-\delta}$. Selecting $\delta$ small enough so that $\frac{\delta}{R-\delta}<\varepsilon$ we get $\Vert f\Vert\leq 1+\varepsilon$. Now, applying the CEP condition, we guarantee the a existence of a norm-preserving extension of $f$ defined on $M$ which satisfies the requirements.\end{proof}

%The previous lemma allows us modify the identity in a ball where it is constant on a smaller ball. The following lemma follows the converse idea, to modify the constant function to act as the identity on a small ball. More precisely:
%
%\begin{lemma}\label{lemma:locidentity}
%Let $M$ be a metric space with the CEP. For every $\varepsilon>0$, for every $m\in M$ and every $R>0$ there exists $\delta>0$ and a Lipschitz function $f:M\longrightarrow M$ so that $\Vert f\Vert\leq 1+\varepsilon$, $f(x)=x$ for every $x\in B(m,\delta)$ and $f(x)=m$ for every $x\notin B(m,R)$.
%\end{lemma}
%
%\begin{proof}
%Similar to the previous.
%\end{proof}
%
%Let us end with a technical lemma which will be useful for our computations.

Now we are ready to give the prending proof.

\begin{proof}[Proof of Theorem \ref{theo:daugaLipgeneral}]
Let $f,g\in S_{\Lip(M,X)}$ and $\varepsilon>0$. Let us find, by (5) in Theorem \ref{theo:cadaDaugavet}, a sequence $\{f_n\}\subseteq \Lip(M,X)$ with $\Vert f_n\Vert\leq 1+\varepsilon$, $\Vert f+f_n\Vert\geq  2(1-\varepsilon)$ and $f_n\rightarrow^w g$.

To this end, as in Theorem \ref{theo:daugaF(M)finCEP}, we can find a sequence of different points $x_n$ so that, for every $\delta>0$ and $n\in\mathbb N$, $\Vert f_{|B(x_n,\delta)}\Vert>1-\varepsilon$. 

Now, for every $n\in\mathbb N$ find $R_n>0$ so that $B(x_n,R_n)$ are pairwise disjoint. We can assume that $0\notin B(x_n,R_n)$ for every $n\in\mathbb N$. Let $0<\eta<\varepsilon$. Given $n\in\mathbb N$ take by Lemma \ref{lema:daugaextindenlocons} $\varphi_n:M\longrightarrow M$ so that $\Vert \varphi\Vert\leq 1+\eta$, that $\varphi(x)=x$ for every $x\notin B(x_n,R_n)$ and that $\varphi=x_n$ on $B(x_n,\mu_n)$ for suitable $\mu_n>0$ which we can assume, with no loss of generality, that $1+\eta+\frac{\mu_n}{R_n-\mu_n}<1+\varepsilon$. 

By the assumptions on the points $x_n$ we can find, for every $n\in\mathbb N$, a pair of different points $u_n,v_n\in M$ so that $\frac{\Vert f(u_n)-f(v_n)\Vert}{d(u_n,v_n)}>1-\varepsilon$ and such that $\max\{\frac{d(u_n,x_n)}{\mu_n-d(u_n,x_n)},\linebreak \frac{d(v_n,x_n)}{\mu_n-d(v_n,x_n)} \}<1$. Now, using Lemma \ref{lemma:finiteceplemalocalF(M)dau}, we can find a norm-one Lipschitz function $\psi_n:M\longrightarrow M$ with $\psi=x_n$ in $M\setminus B(x_n,\mu_n)$, $\psi(x_n)=x_n$, $\psi(u_n)=u_n$ and $\psi(v_n)=v_n$.

 Let $g\circ \varphi_n:M\longrightarrow X^*$, which is a function of norm smaller than or equal to $(1+\eta)$ so that $g\circ\varphi=g$ on $M\setminus B(x_n,R_n)$ and $g\circ\varphi$ is constant on $B(x_n,\mu_n)$. On the other hand, let $h_n:=f\circ\psi_n-f\circ\psi_n(0)\in B_{\Lip(M,X^*)}$, which satisfies that is zero on $M\setminus B(x_n,\mu_n)$ and $h_n(u_n)-h_n(v_n)=(f\circ\psi_n)(u_n)-(f\circ\psi_n)(v_n)=f(u_n)-f(v_n)$. Finally define $f_n:=g\circ\varphi_n+h_n$. Note that $\Vert f_n\Vert\leq 1+\varepsilon$ by Lemma \ref{lemma:compuortolip}, that $\Vert f+f_n\Vert\geq 2(1-\varepsilon)$. Indeed, given $n\in\mathbb N$, we get 
\[
\begin{split}\frac{\Vert f(u_n)-f(v_n)+f_n(u_n)-f_n(v_n)\Vert}{d(u_n,v_n)}& =\frac{\Vert f(u_n)-f(v_n)+f(\psi_n(u_n))-f(\psi_n(v_n))\Vert}{d(u_n,v_n)}\\
& =2\frac{\Vert f(u_n)-f(v_n)\Vert}{d(u_n,v_n)}>2(1-\varepsilon).
\end{split}\]
Finally observe that $\operatorname{supp}(g-f_n)\subseteq B(x_n,R_n)$. Consequently, $g-f_n$ is a bounded sequence so that $\operatorname{supp}(g-f_n)\cap \operatorname{supp}(g-f_m)=\emptyset$ if $n\neq m$. \cite[Lemma 1.5]{ccgmr} implies that $g-f_n$ is weakly null, so $f_n\rightarrow ^w g$, and the proof is done. \end{proof}

It is now time for applications of the theorems of this section. In view of the results obtained in Section \ref{section:CEPandrelated} we start with the following characterisation for subsets of a Hilbert space.

\begin{theorem}\label{theo:daugahilbert}
Let $M$ be a pointed subset of a Hilbert space. The following are equivalent:
\begin{enumerate}
\item $\Lip(M,X)$ has the Daugavet property for every Banach space $X$.
\item $\Lip(M)$ has the Daugavet property.
\item $\mathcal F(M)\pten X$ has the Daugavet property for every Banach space $X$.
\item $\mathcal F(M)$ has the Daugavet property.
\item $M$ is length.
\item $M$ is convex.
\end{enumerate}
\end{theorem}

\begin{proof}
The implications (1)$\Rightarrow$(2) and (3)$\Rightarrow$(4) is immediate. (1)$\Rightarrow$(3) follows because, given a Banach space $X$, $(\mathcal F(M)\pten X)^*=L(\mathcal F(M),X^*)=\Lip(M,X^*)$, and now the implication follows since the Daugavet property passes from a dual Banach space to its predual. (4)$\Rightarrow$(5) is from \cite[Theorem 3.5]{gpr18} whereas (5)$\Rightarrow$(6) is described in Proposition \ref{prop:carasubhilCEP}. Finally, (6)$\Rightarrow$(1) follows since $M$ is geodesic and has the CEP by Proposition \ref{prop:carasubhilCEP}, so Theorem \ref{theo:daugaLipgeneral} applies.
\end{proof}

\begin{remark}\label{remark:mejoraF(H)}
Observe that in \cite[Proposition 3.11]{gpr18} it is proved that $\Lip(M,X)$ has the Daugavet property if $M$ is pointed and lenght and the pair $(M,X)$ satisfies that every Lipschitz function $f:N\subseteq M\longrightarrow X^*$ admits a norm-preserving extension. As main consequence of the above result, it is pointed out that $\mathcal F(H)\pten H$ has the Daugavet property for every Hilbert space $H$. Observe that as a consequence of Theorem \ref{theo:daugahilbert} it is obtained that $\mathcal F(H)\pten X$ has the Daugavet property for every $X$.
\end{remark}

We turn now to giving applications in injectivity assumptions. As we proved in Corollary \ref{coro:finiteCEP}, if $M$ is finitely injective then $M$ has the finite CEP. Clearly, injective metric spaces have the CEP. Taking this into account, we derive from Theorems \ref{theo:daugaF(M)finCEP} and \ref{theo:daugaLipgeneral} the following interesting application to Banach spaces.

\begin{theorem}\label{theo:daugaL1preduals}
Let $X$ be a Banach space.
\begin{enumerate}
\item If $X$ is an $L_1$-predual, then $\mathcal F(X)\pten Y$ has the Daugavet property for every non-zero Banach space $Y$.
\item Moreover, if $X$ is even an injective Banach space, then $\Lip(X,Y)=L(\mathcal F(X),Y)$ has the Daugavet property for every non-zero Banach space $Y$. In particular, $(\mathcal F(X)\pten Y)^*$ (and consequently $\mathcal F(X)\pten Y$ itself) has the Daugavet property for every non-zero Banach space $Y$.
\end{enumerate}
\end{theorem}

Observe that in \cite[Question 1]{gpr18} and in \cite[Remark 3.8]{lr2020} it is asked whether if $\mathcal F(M)$ has the Daugavet property this implies that $\mathcal F(M)\pten X$ or $\Lip(M,X)$ has the Daugavet property for every non-zero Banach space $X$. Theorems \ref{theo:daugaF(M)finCEP} and \ref{theo:daugaLipgeneral} gives a positive answer when the metric space has the (finite) CEP and is geodesic.

Observe that this also gives a partial positive answer to \cite[Question 3.2]{blrohlip}, where it is asked whether the norm of $\mathcal F(M)\pten X$ is octahedral for every non-zero Banach space $X$ whenever $\mathcal F(M)$ is octahedral. Recall that the norm on a Banach space $X$ is said to be  \textit{octahedral} if for every $\varepsilon>0$ and for every finite-dimensional subspace $F$ of $X$ there is some $y$ in the unit sphere of $X$ such that $$\Vert x+\lambda y\Vert\geq (1-\varepsilon)(\Vert x\Vert +\vert \lambda\vert)$$ holds for every $x\in F$ and for every scalar $\lambda$ (see \cite{dgz}).

As we have pointed out, Theorems \ref{theo:daugaF(M)finCEP} and \ref{theo:daugaLipgeneral} provide examples of metric spaces $M$ so that $\mathcal F(M)\pten X$ is octahedral for every Banach space $X$. This is a consequence of the fact that the norm of a Banach space with the Daugavet property is octahedral \cite[Lemma 2.8]{kssw}.

Let us end the paper giving more examples of octahedral norms in spaces of the form $\mathcal F(M)\pten X$. This time, octahedrality condition allows to move the condition on CEP to the target space.

\begin{theorem}\label{theo:OHCEPtarget}
Let $M$ be a metric space so that $M'\neq \emptyset$ and let $Y$ be a non-zero Banach space so that $Y^*$ has the finite-CEP. Then the norm of $\mathcal F(M)\pten Y$ is octahedral.
\end{theorem}

\begin{proof} With no loss of generality, up to changing the origin, we assume that $0\in M'$.

In order to prove that the norm of $\mathcal F(M)\pten Y$ is octahedral take $\mu_1,\ldots, \mu_n\in \mathcal F(M)\pten Y$ and $\varepsilon>0$, and let us find $\mu\in S_{\mathcal F(M)\pten Y}$ so that, for every $1\leq i\leq n$, $\Vert \mu_i+\mu\Vert$ is as close to $1+\Vert\mu_i\Vert$ 
as we wish. This is enough by \cite[Theorem 2.1]{hlp}. To do so, take $\varepsilon>0$. By a density argument we can assume with no loss of generality that $\mu_i:=\sum_{j=1}^{k_i} \delta_{x_{ij}}\otimes y_{ij}$ for certain $y_{ij}\in Y\setminus\{0\}$ and $x_{ij}\in M\setminus\{0\}$ for every $1\leq i\leq n, 1\leq j\leq k_i$ (observe that the assumption on the $x_{ij}$'s holds because $\delta_0$ is the zero vector of $\mathcal F(M)$).
Take $\varepsilon>0$ and, for every $1\leq i\leq n$, $g_i\in S_{\Lip(M,Y^*)}$ so that $g_i(\mu_i)=\sum_{j=1}^{k_i}g_i(x_{ij})(y_{ij})>\Vert \mu_i\Vert-\varepsilon$. 

Take $0<\eta<\varepsilon$ and $1\leq i\leq n$. By Lemma \ref{lemma:fininyextidentidad} we can find $\delta_i>0$ and $\varphi_i:Y^*\longrightarrow Y^*$ so that $\Vert \varphi_i\Vert\leq 1+\eta$, $\varphi_i(g_i(x_{ij}))=g_i(x_{ij})$ holds for $1\leq j\leq k_i$ and $\varphi_i=0$ on $B_{Y^*}(0,\delta_i)$. Given $i\in\{1,\ldots, n\}$ define $h_i:=\frac{\varphi_i\circ g_i}{1+\eta}:M\longrightarrow Y^*$, which satisfies $\Vert h_i\Vert\leq 1$. Moreover,
$$h_i(\mu)=\frac{\sum_{j=1}^{k_i}\varphi_i(g(x_{ij}))(y_{ij})}{1+\eta}=\frac{\sum_{j=1}^{k_i}g_i(x_{ij})(y_{ij})}{1+\eta}>\frac{\Vert \mu_i\Vert-\varepsilon}{1+\eta}.$$ 
Moreover observe that $h_i=0$ on $B(0,\delta_i)$. Indeed, given $x\in B(0,\delta_i)$ then $\Vert g_i(x)\Vert=\Vert g_i(x)-g_i(0)\Vert\leq \Vert g_i\Vert \Vert x\Vert<\delta_i$. This implies that $g_i(x)\in B_{Y^*}(0,\delta_i)$, thus $\varphi_i(g_i(x))=0$.

Take $\delta:=\min\{\delta_1,\ldots, \delta_n\}>0$. By the above $h_i=0$ on $B(0,\delta)$ for every $1\leq i\leq n$. 

Now take $\beta>0$ small enough so that $x_{ij}\notin B(0,\beta)$ for every $i,j$, that $2\beta<\delta$ and that $\frac{\beta}{\delta-\beta}<\varepsilon$, and take a norm-one Lipschitz function $f\in\Lip(M,Y^*)$ so that $f(x)=0$ for every $x\in M\setminus B(0,\beta)$. In order to ensure the existence of such function take $x_0\in B(0,\beta)\setminus\{0\}$, which does exist because $0\in M'$, and define by McShane theorem \cite[Theorem 1.33]{weaver}  a function $u:M\longrightarrow \mathbb R$ so that $u(0)=0$, $u(x)=0$ on $M\setminus B(0,\beta)$ and $u(x_0)=1$. Take $y_0^*\in S_{Y^*}$ and the Lipschitz function $u\otimes y_0^*$ defined by $(u\otimes y_0^*)(x)=u(x)y_0^*$ clearly vanishes out of $B(0,\beta)$ and $(u\otimes y_0^*)(x_0)=y_0^*$, so $u\otimes y_0^*$ is a non-zero Lipschitz function in $\Lip(M,Y^*)$. Now $f:=\frac{u\otimes y_0^*}{\Vert u\otimes y_0^*\Vert}$ satisfies our requirements.

Define $f_i:=h_i+f:M\longrightarrow Y^*$. Observe that Lemma \ref{lemma:compuortolip} implies that $\Vert f_i\Vert\leq 1+\frac{\beta}{\delta-\beta}<1+\varepsilon$. Moreover, since $x_{ij}\notin B(0,\delta)$ for every $i, j$, we get $f(x_{ij})=0$ for every $i,j$. Henceforth
$$f_i(\mu_i)=h_i(\mu_i)>\frac{\Vert \mu_i\Vert-\varepsilon}{1+\eta}.$$
On the other hand, since $\Vert f\Vert=1$, find $x,y\in M$ with $x\neq y$ so that $\frac{\Vert f(x)-f(z)\Vert}{d(x,z)}>1-\varepsilon$. We can assume with no loss of generality that $x\in B(0,\beta)$. Observe that we can also assume that $z\in B(0,2\beta)$, because otherwise $f(z)=0$ and $d(x,z)\geq\beta$, so
$$1-\varepsilon<\frac{\Vert f(x)\Vert}{d(x,z)}\leq \frac{\Vert f(x)\Vert}{d(x,0)}=\frac{\Vert f(x)-f(0)\Vert}{d(x,0)},$$
and in such case we can replace $z$ with $0$. In particular $x,z\in B(0,2\beta)\subseteq B(0,\delta)$ and then $h_i(x)=h_i(z)=0$ holds for every $1\leq i\leq n$. Denote by $\mu:=\frac{\delta_x-\delta_z}{d(x,z)}\otimes y$, which is a norm-one element of $\mathcal F(M)\pten Y$. Observe that $h_i(\mu)=0$ because $x,z\in B(0,\delta)$. Hence, given $1\leq i\leq n$, we get
\[
\Vert \mu_i+\mu\Vert\geq \frac{f_i(\mu_i+\mu)}{\Vert f_i\Vert}\geq \frac{f_i(\mu_i)+f(\mu)}{1+\varepsilon}>\frac{\frac{\Vert \mu_i\Vert-\varepsilon}{1+\eta}+1-\varepsilon}{1+\varepsilon}
\]
Since $\varepsilon$ and $\eta$ can be taken as small as we whish we conclude that the norm of $\mathcal F(M)\pten Y$ is octahedral.\end{proof}

\begin{remark}
In Theorem \ref{theo:OHCEPtarget} we have two possibilites for $Y^*$: either $Y^*$ (and hence $Y$) is a Hilbert space or $Y^*$ is an $L_1$-predual. The latter case does not provide extra information about octahedrality on $\mathcal F(M)\pten Y$. This is because, under the hypothesis on $Y^*$, $Y=L_1(\mu)$ for certain measure $\mu$, so
$$\mathcal F(M)\pten Y=\mathcal F(M)\pten L_1(\mu)=L_1(\mu,\mathcal F(M)),$$
where the last identification can be seen in \cite[Example 2.19]{ryan}. Since the norm of $\mathcal F(M)$ is octahedral \cite[Theorem 2.4]{blrohlip}, then $L_1(\mu,\mathcal F(M))$ always have octahedral norm (see e.g. \cite[P. 852]{llr}). However, to the best of our knownledge, the octahedrality of $\mathcal F(M)\pten H$ when $M'\neq\emptyset$ and $H$ is a Hilbert space was unknown.
\end{remark}

\begin{remark}
Observe that, in the proof of Theorem \ref{theo:OHCEPtarget}, when we define the function $f$ we do not have any control on the pairs $x,y$ at which $f$ approximates its Lipschitz norm. Because of that, we do not know whether the argument can be adapted to obtain that $\mathcal F(M)\pten Y$ has the Daugavet property if we assume on $M$ that it is length (or even geodesic).
\end{remark}

%We finish with the following question derived from our work.
%
%\begin{question} Does $\mathcal F(X)\pten Y$ have the Daugavet property for every pair of Banach spaces $X$ and $Y$?
%\end{question}
%
%Observe that it would be enough to prove that every Banach space $X$ enjoy Lemmata \ref{lemma:fininyextidentidad} and \ref{lemma:finiteceplemalocalF(M)dau}. For instance, for Lemma \ref{lemma:fininyextidentidad} we need to prove that for every $t_1,\ldots, t_n\in X\setminus\{0\}$ and every $\varepsilon>0$, there exists $\delta>0$ and a Lipschitz function $\varphi:X\longrightarrow X$ so that $\varphi(t_i)=t_i$, $\varphi(x)=0$ on $B(0,\delta)$ and $\Vert \varphi\Vert\leq (1+\varepsilon)$. Observe that the later is the most restrictive restriction because, if we forget of this, such $\varphi$ can be easily constructed thanks to the scalar multiplication on $X$.
%
%Indeed, take $0<\delta<R$ so that $\{t_1,\ldots, t_n\}\subseteq M\setminus B(0,R)$ and take a Lipschitz function $f:X\longrightarrow [0,1]$ so that $f(x)=1$  on $M\setminus B(0,R)$ and $f(x)=0$ on $B(0,\delta)$ (such function can be constructed, for instance, as $f(x):=\frac{d(x,B(0,\delta))}{d(x,B(0,\delta))+d(x,M\setminus B(0,R))}$). Now $g(x):=f(x)x$ can be proved to be $2$-Lipschitz and satisfies that $g(t_i)=t_i$ and $g=0$ on $B(0,\delta)$. 

\section*{Acknowledgements}  The author thanks A. Avil\'es, G. Mart\'inez-Cervantes and J. Rodr\'iguez for fruitful conversations on the topic of the paper.

\bibliographystyle{amsplain}

\end{document}